\documentclass[11pt]{article}

\usepackage{amsmath,amssymb,amsthm,mathrsfs}
\usepackage{enumitem}
\usepackage{geometry}
\numberwithin{equation}{section}
\newtheorem{conjecture}{Conjecture}
\newtheorem{theorem}{Theorem}[section]
\newtheorem{lemma}[theorem]{Lemma}
\newtheorem{proposition}[theorem]{Proposition}
\newtheorem{corollary}[theorem]{Corollary}
\newtheorem{remark}[theorem]{Remark}

\DeclareMathOperator*{\argmax}{arg\,max}

\newcommand{\vertk}{\stackrel{{\cal D}}{\longrightarrow}}
\newcommand{\edist}{\stackrel{{\cal D}}{=}}
\title{The location of the largest exponential spacing and Euler’s generalized pentagonal numbers}
\author{Norbert Henze\footnote{Institute of Stochastics, Karlsruhe Institute of Technology (KIT), 76131 Karlsruhe, Germany. e-mail:
henze@kit.edu}}
\date{}

\begin{document}
\maketitle

\begin{abstract}
We study the location of the largest spacing generated by the order
statistics of a sample from the standard exponential distribution.
Although the asymptotic behaviour of the largest spacing itself is well
understood, considerably less is known about the index at which it is
attained. Using the independence and unequal rates of exponential
spacings, we derive exact finite-sample formulas and show that, when
measured relative to the right endpoint, the location of the largest
spacing converges in distribution to a non-degenerate probability
distribution on the positive integers. We obtain explicit integral
representations for the limiting probabilities and, using Euler's
pentagonal number theorem, derive a series representation involving the
generalized pentagonal numbers. This reveals that the Euler product
appearing in the limiting distribution of the largest exponential
spacing also governs the distribution of its location. The convergence
result is also extended to the location of the largest
$m$-spacing for every fixed $m\geq 1$, despite the dependence among
overlapping $m$-spacings. Numerical values
illustrate the concentration of the limiting distribution near the right
endpoint and the rapid convergence of the finite-sample probabilities.
\end{abstract}



\section{Introduction}
The study of spacings generated by order statistics of independent identically distributed random variables with a common Lebesgue density $f$ has a long history
in probability theory (see, e.g., \cite{pyke65}). If $f$ is the density of the uniform distribution on $[0,1]$, asymptotic theory for the largest spacing has been
given by \cite{levy39}, \cite{deh82} and \cite{dev82}, and generalizations for certain non-uniform densities can be found in \cite{deh84}.
When $f$ is the density of an exponential distribution, \cite{devroye84} derived the limit distribution of the largest spacing and obtained corresponding large deviation probabilities
and laws of the iterated logarithm.  In a subsequent paper, \cite{henze87} derived an explicit representation of the limiting distribution of
the largest exponential spacing as a two-sided infinite series involving Euler's pentagonal numbers. While the asymptotic behaviour of the largest
exponential spacing is well understood, little seems to be known about the  index at which this maximum is attained. The present note is devoted to this problem.

To be specific, let
$X_1,\ldots,X_n$ be independent standard exponential random variables, and  write $X_{1:n} \le X_{2:n} \le  \cdots  \le X_{n:n}$
for  the order statistics of $X_1,\ldots,X_n$. Then the spacings $S_{j,n} := X_{j:n}-X_{j-1:n}$ for $j=1,\ldots,n$, where $X_{0:n} :=0$,
form an independent sequence, with $S_{j,n}$ having an exponential distribution with expectation $1/(n-j+1)$, $j=1,\ldots,n$.
This remarkable representation due to \cite{suk37} (see also \cite{pyke65}) makes exponential spacings
a natural and tractable model for the study of extremes of non–identically distributed random variables.
As mentioned above, considerable attention has been devoted to the size of the largest spacing
\[
M_n := \max_{1\le j \le n} S_{j,n}.
\]
While the asymptotic behaviour of $M_n$ is by now well understood,
a closely related and equally natural question appears to have received little explicit attention:
At which position does the largest spacing occur?
More precisely, if
\[ 
I_n := \argmax_{1\le j \le n} S_{j,n},
\] 
what can be said about the distribution of $I_n$
and its asymptotic behaviour as $n \to \infty$?

We show that the position of the largest
exponential spacing admits a non-degenerate limiting distribution when measured relative to the right endpoint. Writing
\begin{equation}\label{defrn}
R_n :=n-I_n+1
\end{equation} for the random position of $I_n$ counted from the largest order statistic, we prove that, for each fixed integer
$r$,
\[
\lim_{n\to \infty} \mathbb{P}(R_n = r) = p_r,
\]
where $(p_r)_{r\ge 1}$  is a genuine probability distribution on the positive integers.

Our approach is probabilistic and relies on a reduction to an infinite array of independent exponential
random variables with increasing rates. This construction shows that the limiting argmax distribution
arises from a well-defined extremal problem for a non-identically distributed sequence,
and that the maximizer exists and is unique almost surely. Explicit integral representations for the probabilities
$p_r$ are derived, allowing for sharp finite-sample approximations.
An interesting by-product of the analysis is that the same infinite product
which governs the limiting distribution of the largest spacing itself (see \cite{devroye84}, \cite{henze87})
also governs the distribution of its maximizer. As a consequence, Euler’s generalized pentagonal numbers reappear
naturally in an explicit series representation of the limiting argmax probabilities.
This shows that the arithmetic structure previously observed for the size of the maximal spacing
 also controls the location at which this maximum is attained.

The paper is organized as follows. Section~\ref{sec:model} reduces the problem
 to an extremal question for independent exponential variables with increasing rates.
 In Section~\ref{sec:finiten} we derive exact finite-sample expressions for the distribution of the maximizer.
 Section~\ref{sec:limit} establishes the existence and normalization of the limiting argmax distribution.
 Analytic representations and the connection with Euler’s generalized pentagonal numbers are developed in Section~\ref{sec:pentagonal}.
 Numerical illustrations of the limiting probabilities and their convergence are presented in Section~\ref{sec:numerics}.
 Section~\ref{sec:discussion} extends the convergence result to fixed $m$-spacings and
formulates a conjecture concerning the mode of the resulting limiting
distribution.


\section{Probabilistic reduction}\label{sec:model}


Throughout the paper we use the notation and assumptions introduced in
Section~1. In particular, the spacings
$S_{1,n},\dots,S_{n,n}$ are independent with
$S_{j,n}\edist \mathrm{Exp}(n-j+1)$, where $\edist$ denotes equality in distribution. 

For analytical convenience, we reverse the indexing and define
\[
Y_r:=S_{n-r+1,n},\qquad r=1,\dots,n.
\]
Then $Y_1,\dots,Y_n$ are independent with $Y_r\edist \mathrm{Exp}(r)$ for $r=1,\dots,n$.

Let $R_n$ as defined in \eqref{defrn}
denote the (almost surely unique) position of the largest spacing. Clearly,
\begin{equation}\label{eq:prn_as_winner}
\mathbb P(R_n = r)  = \mathbb P\!\left(Y_r=\max_{1\le k\le n}Y_k\right),
\qquad 1\le r\le n.
\end{equation}

Representation~\eqref{eq:prn_as_winner} shows that the problem reduces to a
``winner'' question for a triangular array of independent, but non-identically
distributed, exponential random variables with increasing rates. This formulation
will be the basis for the finite-sample analysis in Section~\ref{sec:finiten}.

For the asymptotic theory it is natural to embed $Y_1,\dots,Y_n$ into an infinite
sequence $(Y_r)_{r\ge1}$ of independent exponential random variables satisfying
\begin{equation}\label{eq:infinite_array}
Y_r\edist \mathrm{Exp}(r),\qquad r\ge1.
\end{equation}
The limiting distribution of the maximizer is derived from this infinite-array
construction in Section~\ref{sec:limit}.


\section{Finite-sample distribution of the reverse index}\label{sec:finiten}


In this section we derive exact expressions for the distribution of the reverse index  $R_n$ for fixed $n$. Recall that
$Y_1,\dots,Y_n$  are independent with $Y_r\edist \mathrm{Exp}(r)$,
and that
\begin{equation}\label{gl:defprn}
p_{r,n} :
=\mathbb P\!\left(R_n=r \right)
=\mathbb P\!\left(Y_r=\max_{1\le k\le n}Y_k\right),
\qquad 1\le r\le n.
\end{equation}

\begin{proposition}\label{prop:prn_integral}
For $1\le r\le n$, the probability $p_{r,n}$ admits the integral representation
\begin{equation}\label{eq:prn_integral}
p_{r,n}
=\int_0^\infty r {\rm e}^{-r t}
\prod_{\substack{1\le k\le n\\k\neq r}}
\bigl(1-{\rm e}^{-k t}\bigr)\,{\rm d}t.
\end{equation}
Equivalently,
\begin{equation}\label{eq:prn_u_integral}
p_{r,n}
=\int_0^1 r\,u^{r-1}
\prod_{\substack{1\le k\le n\\k\neq r}}
(1-u^k)\,{\rm d}u.
\end{equation}
\end{proposition}

\begin{proof}
Since the random variables $Y_1,\dots,Y_n$ are independent and continuously
distributed, conditioning on the value of $Y_r$ yields
\[
p_{r,n}
=\int_0^\infty f_r(t)\,
\mathbb P\!\left(Y_k<t \text{ for all } k\neq r\right)\,{\rm d}t,
\]
where $f_r(t)=r {\rm e}^{-r t}$, $t >0$,  is the density of $Y_r$.
Using independence and $\mathbb P(Y_k<t)=1-{\rm e}^{-k t}$ for $k\neq r$, we obtain
\eqref{eq:prn_integral}. The transformation $u={\rm e}^{-t}$ gives
\eqref{eq:prn_u_integral}.
\end{proof}

It is convenient to express~\eqref{eq:prn_u_integral} in terms of the function
\[
\Phi_n(u):=\prod_{k=1}^n(1-u^k),\qquad 0<u<1,
\]
which converges pointwise to Euler's product $\prod_{k=1}^\infty (1-u^k)$ as $n \to \infty$.

\begin{corollary}\label{cor:prn_phi}
For $1\le r\le n$,
\begin{equation}\label{eq:prn_phi}
p_{r,n}
=\int_0^1 r\,u^{r-1}\frac{\Phi_n(u)}{1-u^r}\,{\rm d}u.
\end{equation}
\end{corollary}

\begin{proof}
Since
\[
\prod_{\substack{1\le k\le n\\k\neq r}}(1-u^k)
=\frac{\prod_{k=1}^n(1-u^k)}{1-u^r}
=\frac{\Phi_n(u)}{1-u^r},
\]
equation~\eqref{eq:prn_phi} follows immediately from
\eqref{eq:prn_u_integral}.
\end{proof}

As a consistency check, we note that the values $\{p_{r,n}:1\le r\le n\}$
form a probability distribution.

\begin{remark}\label{rem:normalization_finite}
Since with probability one there exists a unique maximizer among
$Y_1,\dots,Y_n$, it follows that
$
\sum_{r=1}^n p_{r,n}=1$.
\end{remark}

Representations~\eqref{eq:prn_integral}--\eqref{eq:prn_phi} will serve as the
starting point for the asymptotic analysis carried out in the next section.


\section{Limiting distribution of the reverse index}\label{sec:limit}

In this section, we derive the limiting distribution of the reverse index $R_n$. The analysis is based on an infinite-array embedding of the
finite-sample model introduced in Section~\ref{sec:model}. To this end,
let $(Y_r)_{r\ge1}$ be a sequence of independent random variables satisfying \eqref{eq:infinite_array}.
For $r\ge1$, define
\begin{equation}\label{eq:pr_def}
p_r:=\mathbb P\!\left(Y_r=\max_{k\ge1}Y_k\right).
\end{equation}

We first show that the maximization problem on the right-hand side of
\eqref{eq:pr_def} is well posed.

\begin{lemma}\label{lem:existence_unique}
With probability one, there exists a unique index $R\geq 1$ such that
\[
Y_R=\max_{k\geq 1}Y_k<\infty.
\]
\end{lemma}

\smallskip
\begin{proof}
For every $m\geq 1$,
\[
\sum_{k=1}^{\infty} \mathbb P\!\left(Y_k>\frac1m\right) = \sum_{k=1}^{\infty}{\rm e}^{-k/m} <\infty.
\]
By the first Borel--Cantelli lemma, for each fixed $m\ge1$, with
probability one, $Y_k>1/m$ occurs for only finitely many indices $k$.
Since $m$ ranges over a countable set, it follows that, almost surely,
for every $m\ge 1$ there exists $K_m<\infty$ such that
$Y_k\le 1/m$ for all $k\ge K_m$.
Consequently, $Y_k\to 0$ almost surely.

Since $Y_1>0$ almost surely, it follows that, with probability one, there
exists a finite random integer $N$ such that
\[
Y_k<\frac{Y_1}{2},\qquad k>N.
\]
Consequently,
\[
\sup_{k\geq1}Y_k=\max_{1\leq k\leq N}Y_k,
\]
so the supremum is finite and is attained.

Finally, since the random variables are independent and absolutely
continuous, it follows that $\mathbb P(Y_r=Y_s)=0$, $r\neq s$.
Thus the maximizer is unique almost surely.
\end{proof}

\medskip

Lemma~\ref{lem:existence_unique} shows that the events
\[
A_r:=\Big\{Y_r=\max_{k\ge1}Y_k\Big\},\qquad r\ge1,
\]
form a partition of the probability space up to a null set, and that the
probabilities $(p_r)_{r\ge1}$ are well defined.
Recall $p_{r,n}$ from \eqref{gl:defprn}.

\medskip
\begin{theorem}\label{thm:pr_convergence}
For each fixed $r\ge1$,
\begin{equation}\label{eq:pr_limit}
p_{r,n}\to p_r
\qquad \text{as } n\to\infty.
\end{equation}
Moreover, $p_r>0$ for every $r\geq 1$, and
\begin{equation}\label{eq:sum_pr}
\sum_{r=1}^\infty p_r = 1.
\end{equation}
\end{theorem}
\smallskip
\begin{proof}
Fix $r\ge1$. By Proposition~\ref{prop:prn_integral} and its infinite analogue,
\[
p_{r,n}
=\int_0^\infty  r {\rm e}^{-r t}
\prod_{\substack{1\le k\le n\\k\neq r}}
\bigl(1- {\rm e}^{-k t}\bigr)\,{\rm d}t,
\qquad
p_r
=\int_0^\infty  r {\rm e}^{-r t}
\prod_{\substack{k\ge1\\k\neq r}}
\bigl(1- {\rm e}^{-k t}\bigr)\,{\rm d}t.
\]
For each $t>0$, the finite product decreases monotonically in $n$ to the infinite
product. Since
\[
0\le  r{\rm e}^{-r t}
\prod_{\substack{1\le k\le n\\k\neq r}}
(1-{\rm e}^{-k t})
\le  r{\rm e}^{-r t},
\]
and $r {\rm e}^{-r t}$ is integrable on $(0,\infty)$, dominated convergence yields
\eqref{eq:pr_limit}.

Furthermore, for every $t>0$,
\[
\sum_{\substack{k\geq 1\\ k\neq r}} {\rm e}^{-kt}<\infty,
\]
and hence
\[
\prod_{\substack{k\geq 1\\ k\neq r}}
(1-{\rm e}^{-kt})>0.
\]
It follows from the integral representation above that $p_r>0$.

To prove~\eqref{eq:sum_pr}, note that by Lemma~\ref{lem:existence_unique}, with
probability one there exists a unique index $R\ge1$ such that
$Y_R=\max_{k\ge1}Y_k$. Hence
\[
\sum_{r=1}^\infty p_r
=\sum_{r=1}^\infty \mathbb P(A_r)
=\mathbb P\!\left(\bigcup_{r\ge1} A_r\right)
=1.
\]
\end{proof}

\begin{remark}\label{rem:interpretation}
Let $R$ be the almost surely unique index from Lemma~\ref{lem:existence_unique}. Then
\[
p_r=\mathbb{P}(R=r), \qquad r\geq 1.
\]
In view of \eqref{eq:pr_limit}  and \eqref{eq:sum_pr},
\[
R_n \vertk R.
\]
Thus the limiting distribution has full support on the positive
integers and, in particular, is non-degenerate.
\end{remark}


\section{Analytic representations and Euler's pentagonal numbers}\label{sec:pentagonal}

We now derive analytic representations for the limiting probabilities
$(p_r)_{r\ge1}$ and show that Euler's generalized pentagonal numbers arise naturally in this
context. Recall $p_r$ from \eqref{eq:pr_def}. By conditioning
on $Y_r$ and arguing as in Proposition~\ref{prop:prn_integral}, we obtain the integral
representation
 \[
p_r
=\int_0^\infty r {\rm e}^{-r t}
\prod_{\substack{k\ge1\\k\neq r}}
\bigl(1-{\rm e}^{-k t}\bigr)\,{\rm d}t.
\]
With the substitution $u={\rm e}^{-t}$, this becomes
\begin{equation}\label{eq:pr_u_integral}
p_r
=\int_0^1 r\,u^{r-1}
\prod_{\substack{k\ge1\\k\neq r}}(1-u^k)\,{\rm d}u.
\end{equation}

Introduce Euler's product
\[
\Phi(u):=\prod_{k=1}^\infty (1-u^k), \qquad 0<u<1.
\]
Since
\[
\prod_{\substack{k\ge1\\k\neq r}}(1-u^k)
=\frac{\Phi(u)}{1-u^r},
\]
representation~\eqref{eq:pr_u_integral} can be written in the compact form
\begin{equation}\label{eq:pr_phi}
p_r
=\int_0^1 r\,u^{r-1}\frac{\Phi(u)}{1-u^r}\,{\rm d}u,
\qquad r\ge1.
\end{equation}

Euler's pentagonal number theorem (see, e.g., \cite{and76}, Corollary 1.7) yields the expansion
\begin{equation}\label{eq:euler_pentagonal}
\Phi(u)
=1+\sum_{m=1}^{\infty}(-1)^m
\bigl(u^{g_m^-}+u^{g_m^+}\bigr),
\qquad 0<u<1,
\end{equation}
where
\[
g_m^-=\frac{m(3m-1)}{2},
\qquad
g_m^+=\frac{m(3m+1)}{2},
\]
are the generalized pentagonal numbers.

Moreover, for $r\ge1$,
\begin{equation}\label{eq:geometric}
\frac{1}{1-u^r}=\sum_{\ell=0}^\infty u^{\ell r},
\qquad 0<u<1.
\end{equation}
Substituting \eqref{eq:euler_pentagonal} and \eqref{eq:geometric} into
\eqref{eq:pr_phi} and integrating termwise yields the following explicit series
representation.


\begin{proposition}\label{prop:pr_pentagonal}
For $r\geq1$,
\begin{equation}\label{eq:pr_pentagonal}
p_r
=
r\sum_{\ell=0}^{\infty}\Biggl[
\frac{1}{r(\ell+1)}
+\sum_{m=1}^{\infty}(-1)^m
\left(
\frac{1}{r(\ell+1)+g_m^-}
+\frac{1}{r(\ell+1)+g_m^+}
\right)
\Biggr].
\end{equation}
Here, for each fixed $\ell$, the inner series with respect to $m$ is
evaluated first.
\end{proposition}

\begin{proof}
Using the geometric expansion \eqref{eq:geometric} in
\eqref{eq:pr_phi}, we obtain
\[
p_r
=
r\int_0^1
u^{r-1}\Phi(u)
\sum_{\ell=0}^{\infty}u^{\ell r}\,{\rm d}u.
\]
Since $\Phi(u)\geq0$ for $0<u<1$, the monotone convergence theorem
yields
\[
p_r
=
r\sum_{\ell=0}^{\infty}
\int_0^1 u^{r(\ell+1)-1}\Phi(u)\,{\rm d}u.
\]

Fix $\ell\geq0$ and put
\[
a_\ell:=r(\ell+1).
\]
By Euler's pentagonal number theorem,
\[
u^{a_\ell-1}\Phi(u)
=
u^{a_\ell-1}
+
\sum_{m=1}^{\infty}(-1)^m
\left(
u^{a_\ell+g_m^- -1}
+
u^{a_\ell+g_m^+ -1}
\right).
\]
Moreover,
\begin{align*}
\sum_{m=1}^{\infty}
\int_0^1
u^{a_\ell-1}
\left(
u^{g_m^-}+u^{g_m^+}
\right)\,{\rm d}u
&=
\sum_{m=1}^{\infty}
\left(
\frac{1}{a_\ell+g_m^-}
+
\frac{1}{a_\ell+g_m^+}
\right)\\
&\leq
2\sum_{m=1}^{\infty}\frac{1}{m^2}
<\infty,
\end{align*}
where we have used $g_m^-\geq m^2$ and $g_m^+\geq m^2$.
Consequently, Fubini's theorem permits termwise integration of the
pentagonal expansion, and hence
\[
\int_0^1u^{a_\ell-1}\Phi(u)\,{\rm d}u
=
\frac{1}{a_\ell}
+
\sum_{m=1}^{\infty}(-1)^m
\left(
\frac{1}{a_\ell+g_m^-}
+
\frac{1}{a_\ell+g_m^+}
\right).
\]
Substituting $a_\ell=r(\ell+1)$ into the preceding representation and
then summing over $\ell$ gives \eqref{eq:pr_pentagonal}.
\end{proof}

\begin{remark}\label{rem:pentagonal_interpretation}
Representation~\eqref{eq:pr_pentagonal} shows that the generalized pentagonal
numbers govern the distribution of the maximizer of the largest exponential
spacing. The same Euler product $\Phi(u)$ appears in the limiting distribution of
the maximal spacing itself; cf.\ \cite{henze87}. Thus, the arithmetic structure
previously observed for the \emph{size} of the largest spacing also controls the
\emph{location} at which this maximum is attained.
\end{remark}

\section{Numerical illustration}\label{sec:numerics}

In this section we provide a brief numerical illustration of the limiting
distribution $(p_r)_{r\ge1}$ and of the convergence of the finite-sample
probabilities $p_{r,n}$ to their limits.

The integral representation~\eqref{eq:pr_phi} and the pentagonal expansion
\eqref{eq:pr_pentagonal} allow for accurate numerical evaluation of the limiting
probabilities $p_r$. Table~\ref{tab:pr_values} lists approximate values of $p_r$ for
$r=1,\dots,10$.

For Table~\ref{tab:pr_values}, the infinite product in \eqref{eq:pr_u_integral} was truncated at $k=100$.
Repeating the computation with truncation at $k=200$ changed each
displayed probability by less than $2\times10^{-10}$.

\begin{table}[h]
\centering
\begin{tabular}{c c @{\hspace{8mm}} c c}
\hline
$r$ & $p_r$ & $r$ & $p_r$ \\[1mm]
\hline
1 & 0.5161 & 6  & 0.0221 \\
2 & 0.2132 & 7  & 0.0142 \\
3 & 0.1073 & 8  & 0.0094 \\
4 & 0.0598 & 9  & 0.0064 \\
5 & 0.0355 & 10 & 0.0044 \\
\hline
\end{tabular}
\caption{Approximate values of the limiting probabilities $p_r$.}
\label{tab:pr_values}
\end{table}

These values were obtained from the exact integral representation
\eqref{eq:pr_u_integral} by truncating the infinite product and evaluating the resulting integral numerically.
The truncation error is negligible at the displayed precision; all values are accurate to at least four decimal places.

The distribution is clearly concentrated on small values of $r$, indicating that
the largest spacing is asymptotically most likely to occur among the last few
spacings near the right endpoint. In particular, $p_1\approx 0.516$,
so that with probability slightly larger than one half the largest spacing is the
final spacing.

To illustrate the convergence $p_{r,n}\to p_r$, Table~\ref{tab:prn_values} reports
finite-sample probabilities for $n=10$ and $n=20$, computed from the exact integral
representation~\eqref{eq:prn_phi}.

\begin{table}[h]
\centering
\begin{tabular}{c c c c}
\hline
$r$ & $p_{r,10}$ & $p_{r,20}$ & $p_r$ \\[1mm]
\hline
1 & 0.5184 & 0.5162 & 0.5161 \\
2 & 0.2152 & 0.2133 & 0.2132 \\
3 & 0.1089 & 0.1074 & 0.1073 \\
4 & 0.0611 & 0.0598 & 0.0598 \\
5 & 0.0366 & 0.0356 & 0.0355 \\
\hline
\end{tabular}
\caption{Finite-sample probabilities $p_{r,n}$ and their limits $p_r$.}
\label{tab:prn_values}
\end{table}

Even for moderate sample sizes, the agreement between $p_{r,n}$ and $p_r$ is already
remarkably close. This illustrates the rapid numerical convergence in this example
and confirms that the limiting distribution provides
an accurate description of the location of the largest spacing for practical sample
sizes.


\section{Extension to fixed $m$-spacings}\label{sec:discussion}


The preceding sections concern ordinary spacings, corresponding to the case
$m=1$. We now show that the infinite-sequence approach also yields a limiting
distribution for the location of the largest fixed $m$-spacing. Although
overlapping $m$-spacings are dependent when $m\geq 2$, this dependence does not
obstruct the existence of a limiting reverse-index distribution.

For a fixed integer $m\geq 1$ and $n\geq m$, define the $m$-spacings by
\[
S^{(m)}_{j,n} := X_{j:n}-X_{j-m:n} = \sum_{\ell=0}^{m-1} S_{j-\ell,n}, \qquad j=m,\dots,n,
\]
and let
\[
I_n^{(m)} := \operatorname*{arg\,max}_{j=m,\dots,n} S^{(m)}_{j,n}
\]
denote the almost surely unique position of the largest $m$-spacing. As in
the case $m=1$, we measure its position from the right endpoint and put
\begin{equation}
R_n^{(m)} := n-I_n^{(m)}+1.
\label{eq:Rnm}
\end{equation}
Thus $R_n^{(m)}$ takes values in $\{1,\dots,n-m+1\}$. Notice that
\[
S^{(1)}_{j,n}=S_{j,n}, \qquad I_n^{(1)}=I_n, \qquad R_n^{(1)}=R_n.
\]
Exponential $m$-spacings have been studied, among others, by \cite{bsn23} and \cite{rifi18}, while asymptotic results for lower
exponential spacings are given by \cite{best20}. The limiting distribution of the location of the largest
 $m$-spacing appears not to have been studied
for $m\geq 2$.

Let $(Y_r)_{r\geq 1}$ be the independent sequence introduced in \eqref{eq:infinite_array},
so that $Y_r\edist \operatorname{Exp}(r)$, and define
\begin{equation}
Z_r^{(m)} := \sum_{\ell=0}^{m-1}Y_{r+\ell}, \qquad r\geq 1.
\label{eq:Zrm}
\end{equation}

\medskip

\begin{theorem}\label{thm:mspacings}
Let $m\geq 1$ be fixed. With probability one, there exists a unique random index
\[
R^{(m)} := \operatorname*{arg\,max}_{r\geq 1} Z_r^{(m)}.
\]
Moreover,
\[
R_n^{(m)} \vertk R^{(m)} \qquad\text{as } n\to\infty.
\]
The limiting distribution has full support on the positive integers. More
precisely, if
\[
q_r^{(m)} := \mathbb{P}\bigl(R^{(m)}=r\bigr), \qquad r\geq 1,
\]
then
\[
q_r^{(m)}>0 \quad\text{for every }r\geq 1, \qquad \sum_{r=1}^{\infty}q_r^{(m)}=1.
\]
\end{theorem}

\smallskip

\begin{proof}
For $1\leq r\leq n-m+1$,
\[
S^{(m)}_{n-r+1,n} = \sum_{\ell=0}^{m-1}S_{n-r+1-\ell,n}.
\]
Consequently, the exponential-spacing representation gives the joint distributional identity
\[
\left(S^{(m)}_{n-r+1,n} \right)_{1\leq r\leq n-m+1} \edist \left(Z_r^{(m)}\right)_{1\leq r\leq n-m+1}.
\]
It follows that
\[
R_n^{(m)} \edist \widetilde R_n^{(m)}
:= \operatorname*{arg\,max}_{1\leq r\leq n-m+1}Z_r^{(m)}.
\]

For every integer $j\geq 1$,
\[
\sum_{k=1}^{\infty} \mathbb{P}\left(Y_k>\frac{1}{j}\right) = \sum_{k=1}^{\infty}{\rm e}^{-k/j} < \infty.
\]
The first Borel--Cantelli lemma therefore implies that
$Y_k \to 0$ almost surely. Since $m$ is fixed, it follows that
\[
Z_r^{(m)} = \sum_{\ell=0}^{m-1}Y_{r+\ell} \to 0 \qquad\text{almost surely as }r\to\infty.
\]
On the other hand, $Z_1^{(m)}>0$ almost surely. Hence, with probability one, there exists a finite random integer $N$ such that
\[
Z_r^{(m)}<\frac{Z_1^{(m)}}{2}, \qquad r>N.
\]
Thus the supremum of $(Z_r^{(m)})_{r\geq 1}$ is attained among finitely many indices.

To prove uniqueness, let $r<s$. The variable $Y_r$ occurs in $Z_r^{(m)}$
but not in $Z_s^{(m)}$. Conditioning on all variables other than $Y_r$, and using the absolute continuity of $Y_r$, gives
\[
\mathbb{P}\left(Z_r^{(m)}=Z_s^{(m)}\right)=0.
\]
A countable union over all pairs $r<s$ shows that the maximizer is unique almost surely.

Since $R^{(m)}$ is finite almost surely, we have  $\widetilde R_n^{(m)}=R^{(m)}$
whenever $n-m+1\geq R^{(m)}$. Consequently, $\widetilde R_n^{(m)} \to R^{(m)}$ almost surely,
and the distributional identity above yields $R_n^{(m)} \vertk R^{(m)}$.

It remains to prove that the limiting distribution has full support. Fix
$r\geq 1$ and choose numbers $0<\varepsilon<a$. Put
$B_r:=\{r,\dots,r+m-1\}$ and consider the event
\[
E_r := \left\{ Y_k>a \text{ for every }k\in B_r \right\} \cap \left\{ Y_k<\varepsilon \text{ for every }k\notin B_r \right\}.
\]
By independence,
\[
\mathbb{P}(E_r) = \prod_{k\in B_r}{\rm e}^{-ka}
\prod_{\substack{k\geq 1\\ k\notin B_r}} \left(1-{\rm e}^{-k\varepsilon}\right) >0,
\]
because
\[
\sum_{\substack{k\geq 1\\ k\notin B_r}} {\rm e}^{-k\varepsilon} < \infty.
\]

For $s\neq r$, let $B_s:=\{s,\dots,s+m-1\}$. On $E_r$, the common terms in $Z_r^{(m)}$ and $Z_s^{(m)}$ cancel, and
\[
Z_r^{(m)}-Z_s^{(m)}
 = \sum_{k\in B_r\setminus B_s}Y_k - \sum_{k\in B_s\setminus B_r}Y_k
 > \lvert B_r\setminus B_s\rvert(a-\varepsilon) > 0.
\]
Thus $E_r\subseteq\{R^{(m)}=r\}$, and hence $q_r^{(m)}>0$. Finally, $\sum_{r\geq 1}q_r^{(m)}=1$ follows from the almost sure existence and
uniqueness of $R^{(m)}$.
\end{proof}

Theorem~\ref{thm:mspacings} settles the existence of the limiting
distribution. The dependence among overlapping $m$-spacings does, however,
make an explicit characterization of the probabilities $q_r^{(m)}$
substantially more difficult than in the case $m=1$.

For $m=2$ and $m=3$, we simulated the distribution of $R_n^{(m)}$
for sample sizes $n=20$ and $n=50$, using $2\times 10^7$ independent
Monte Carlo replications for each setting. For $m=2$, the estimated
values of $\mathbb{P}(R_n^{(2)}=1)$ were $0.6449$ and $0.6448$ for
$n=20$ and $n=50$, respectively. For $m=3$, the corresponding
estimates were $0.7221$ and $0.7221$. The Monte Carlo standard errors
were at most $1.1\times 10^{-4}$.

These observations motivate the following conjecture.

\begin{conjecture}\label{conj:mspacings}
For every fixed integer $m\geq 2$, the limiting distribution of
$R^{(m)}$ has a mode at $1$; that is,
\[
q_1^{(m)}
=
\max_{r\geq 1}q_r^{(m)}.
\]
\end{conjecture}

It would be of interest to obtain explicit representations for the probabilities $q_r^{(m)}$, analogous to those obtained in Section~\ref{sec:pentagonal} for
$m=1$, and to establish Conjecture~\ref{conj:mspacings}. We hope that the present results provide a useful complement to the existing literature on
exponential spacings and stimulate further work on extremal properties of inhomogeneous random arrays.

\end{document}